\documentclass[12pt]{article}
\usepackage[usenames]{color}
\usepackage{graphicx, subfigure}
\usepackage{amsmath, amsthm, amssymb}
\usepackage{amsfonts}
\usepackage{fullpage}
\usepackage{ifthen}
\usepackage{url}
\usepackage[sort&compress]{natbib}
\usepackage{multirow}

\usepackage[linesnumbered,algoruled,boxed,lined,commentsnumbered]{algorithm2e}
\usepackage{bm}
\usepackage{tikz-qtree}
\usepackage{tikz}
\usetikzlibrary{automata,arrows,positioning,calc}
\usepackage{hvlogos}

\newtheorem{theorem}{Theorem}[section]

\makeatletter
\def\@biblabel#1{}
\makeatother

\theoremstyle{plain}
\newtheorem{lem}{Lemma}

\theoremstyle{definition}
\newtheorem{example}{Example}

\theoremstyle{remark}

\newcommand{\E}{\mathsf{E}}

\newcommand{\ber}{{\sf Ber}}

\title{On the problem of assigning multiple interceptors over multiple aerial threats}

\author{Liang Hong\footnote{Department of Mathematical Sciences, The University of Texas at Dallas, 800 West Campbell Road, Richardson, TX 75080, USA. Tel.:~+1 (972) 883-2161. Email address: liang.hong@utdallas.edu.}}
\date{\today}

\begin{document}

\maketitle

\begin{abstract}
This article investigates the problem of assigning multiple identical interceptors over multiple identical aerial threats,  where all interceptors are launched in a single salvo.  For this problem, two strategies have been studied in the literature: (A) to spread all interceptors as evenly as possible over all threats, and (B) to randomly assign all interceptors over the threats.  The main contributions of this article are as follows.  First,   the literature contains empirical evidence that Strategy~A is more efficient than Strategy~B in terms of the mean number of missed threats, when the number of interceptors is no less than the number of threats.  This article gives a rigorous proof of this fact.  Secondly, it demonstrates numerically that Strategy~A can be significantly more efficient than Strategy~B.  Thirdly,  it shows that Strategy~A is not only superior to Strategy~B,  but also the bona fide optimal strategy.  Finally, this article establishes that these conclusions also hold when the number of interceptors is less than the number of threats.

\smallskip

\end{abstract}

\section{Introduction}

The last few years have witnessed an increasing number of military conflicts around the world.  A conspicuous feature in these conflicts is the heavy use of drones and missiles.  As a result, a country must increase its capability to neutralize aerial threats.  To this end,  a country can make efforts in two directions: (i) to improve the quality of its interceptors, and (ii) to optimize its assignment of interceptors during a combat.  The former falls within the domain of weapon engineering; the latter is the business of operations research analysts.  The article is exclusively concerned with the latter.  

In this article, we investigate the problem of assigning $m$ identical interceptors over $n$ aerial threats where $m\geq 2$, $n\geq 2$, and each interceptor hits an aerial threat with probability $0<p<1$.   We assume that all interceptors are fired independently but in a single salvo (so that no feedback is available between the firing of any two interceptors).  We are interested in two strategies: 
\begin{enumerate}
\item[(A)]spread all interceptors as evenly as possible over all threats;
\item[(B)]randomly assign all interceptors over the threats. 
\end{enumerate}
We will use the mean number of missed threats to assess the efficiency of a given strategy.  This is a classical problem in the context of assigning identical shots to identical targets, and it is well-documented in the literature (e.g.,   Sections~2.4.1 and 2.4.4 of Washburn and Kress 2009).  Though research in the past three decades has been focusing on more general cases where decoys are present, weapons or targets are not identical,  or feedback is available (e.g., Aviv and Kress 1997; Yost and Washburn 2000; Glazebrook and Washburn 2004; Washburn 2005; Pryluk and Shima 2014; Kline et al. 2019; Kalyanam and Clarkson 2021; Atkinson and Kress 2025),  this classical problem has not been completed settled yet.

Let $X_A$ and $X_B$ be the numbers of missed threats when the interceptors are fired using Strategy~A and Strategy~B,  respectively.   In combat, the number of interceptors can be either less than the number of threats (i.e., $m<n$) or the other way around (i.e., $m\geq n$).  Therefore,  we need to examine both cases.  To my knowledge, the case $m<n$ has not been investigated.  The case $m\geq n$ has been studied,  and the key results are summarized in Section~2.4.1 and Section~2.4.4 of Washburn and Kress (2009).  In particular, it is known that 
\begin{eqnarray}
\label{eq:mean}
\E[X_A] &=& (n-r)(1-p)^k+r(1-p)^{k+1}, \nonumber \\
\E[X_B] &=& n(1-p/n)^m, 
\end{eqnarray}
where $m=nk+r$, $k$ is a positive integer,  $r$ is an integer such that $0\leq r<n$.   It has also been demonstrated empirically that Strategy~A is more efficient than Strategy~B,  i.e., $\E[X_A]\leq \E[X_B]$.  However,  several important questions remain open. 
\begin{enumerate}
\item[(1)]Does $\E[X_A]\leq \E[X_B]$ always hold? 
\item[(2)]Is Strategy~A only slightly more efficient than Strategy~B or it can be significantly more efficient than Strategy~B?
\item[(3)]Is Strategy~A the optimal strategy in terms of the mean number of missed threats?
\item[(4)]How about the case $m<n$?
\end{enumerate}
This article answers these questions by establishing the following facts.
\begin{enumerate}
\item[$\bullet$]$\E[X_A]\leq \E[X_B]$ always holds.
\item[$\bullet$]Strategy~A can be significantly more efficient than Strategy~B.
\item[$\bullet$]Strategy~A is the optimal strategy in terms of the mean number of missed threats. 
\item[$\bullet$]The above three facts hold for the case $m<n$. 
\end{enumerate}
These results complement the extant results.  Together, they give a fairly complete treatment of the problem of assigning multiple identical shots to multiple identical targets without feedback.

\section{Main results}

Throughout,  we will apply the following inequality a few times. 

\begin{lem}[Bernoulli's inequality]
Let $x\in\mathbb{R}$ and $b\in\mathbb{N}$. For all $x\geq -1$, $x\neq 0$,  and $b\geq 1$,  
\[
(1+x)^b\geq 1+bx.
\]
\end{lem}

\begin{proof}
See,  for instance, Page~26 of Stromberg (1981) or Page~31 of Steele (2008).
\end{proof}

\subsection{When $m\geq n$}

\begin{theorem}
\label{thm:comparison}
If $ m\geq n\geq 2$ and $0<p<1$,  then $\E[X_A]\leq \E[X_B]$.
\end{theorem}

\begin{proof}
In this case, we have $m=nk+r$,  where $k$ is a positive integer and $r$ is an integer such that $0\leq r<n$.  It follows from (\ref{eq:mean}) that 
\begin{eqnarray*}
\frac{\E[X_B]}{\E[X_A]} &=& \frac{n(1-p/n)^m}{(n-r)(1-p)^k+r(1-p)^{k+1}}\\
				  &=&\frac{(1-p/n)^{nk+r}}{(1-p)^k(1-rp/n)}\\
				  &=& \left[\frac{(1-p/n)^r}{(1-rp/n)}\right]\left[\frac{(1-p/n)^n}{(1-p)}\right]^k.
\end{eqnarray*}
Beroulli's inequality  implies $(1-p/n)^r\geq 1-rp/n$ and $(1-p/n)^n\geq 1-p$. Therefore, 
\[
\frac{\E[X_B]}{\E[X_A]} \geq \left[\frac{(1-rp/n)}{(1-rp/n)}\right]\left[\frac{(1-p)}{(1-p)}\right]^k=1.
\]
Since $\E[X_A]$ and $\E[X_B]$ are both positive., we have $\E[X_A]\leq \E[X_B]$.
\end{proof}

Define 
\[
D=\E[X_B]/n-\E[X_A]/n.
\]
That is, $D$ is the difference between the mean numbers of missed threats per threat using Strategy~A and Strategy~B.  Since $\E[X_A]/n$ and $\E[X_B]/n$ are both between $0$ and $1$,  Theorem~\ref{thm:comparison} implies that $0\leq D\leq 1$. Intuitively, $D$ measures the efficiency we gain in using Strategy~A over Strategy~B.  When $p=0.4$ and $n\geq 100$,  Washburn and Kress (2009) find that $D$ is fairly small for all $k\geq 1$, i.e., using Strategy~A does not gain much more than using Strategy~B.  This begs the question: Can $D$ be significant in some cases? The next example shows that the answer is affirmative. 

\begin{example}
\label{ex1}
Suppose $p=0.8$, $n=10$,  $k=2$ and $r=1$. Then $\E[X_A]/n=0.0368$, $\E[X_B]/n=0.1736$,  $\E[X_A]/\E[X_B]=21.20\%$,  and $D=0.137$.  In this case,  Strategy~A is nearly $5$ times more efficient than Strategy~B.  In view of the range and interpretation of $D$,  the value of $D$ is significant.   Panel~(a) of Figure~\ref{fig:ex1} provides plots of $\E[X]/n$ for Strategy~A (circle) and Strategy~B (bullet) when $k$ ranges from $1$ to $10$.   It shows that the advantage of Strategy~A over Strategy~B becomes less pronounced when $k$ increases.  Intuitively,  as $k$ increases,  more interceptors are likely to be assigned to a threat even if we randomly assign all interceptors; hence,  that threat is likely to be neutralized.   Panel~(b) gives the plot of $D$ as $p$ increases, and it suggests that when the quality of the interceptor is relatively high (i.e.,  the success probability $p$ is relatively large),  the difference between the two strategies tends to be significant.  

\begin{figure}[h]
\begin{center}
\subfigure[Mean number of missed threats per threat]{\scalebox{0.45}{\includegraphics{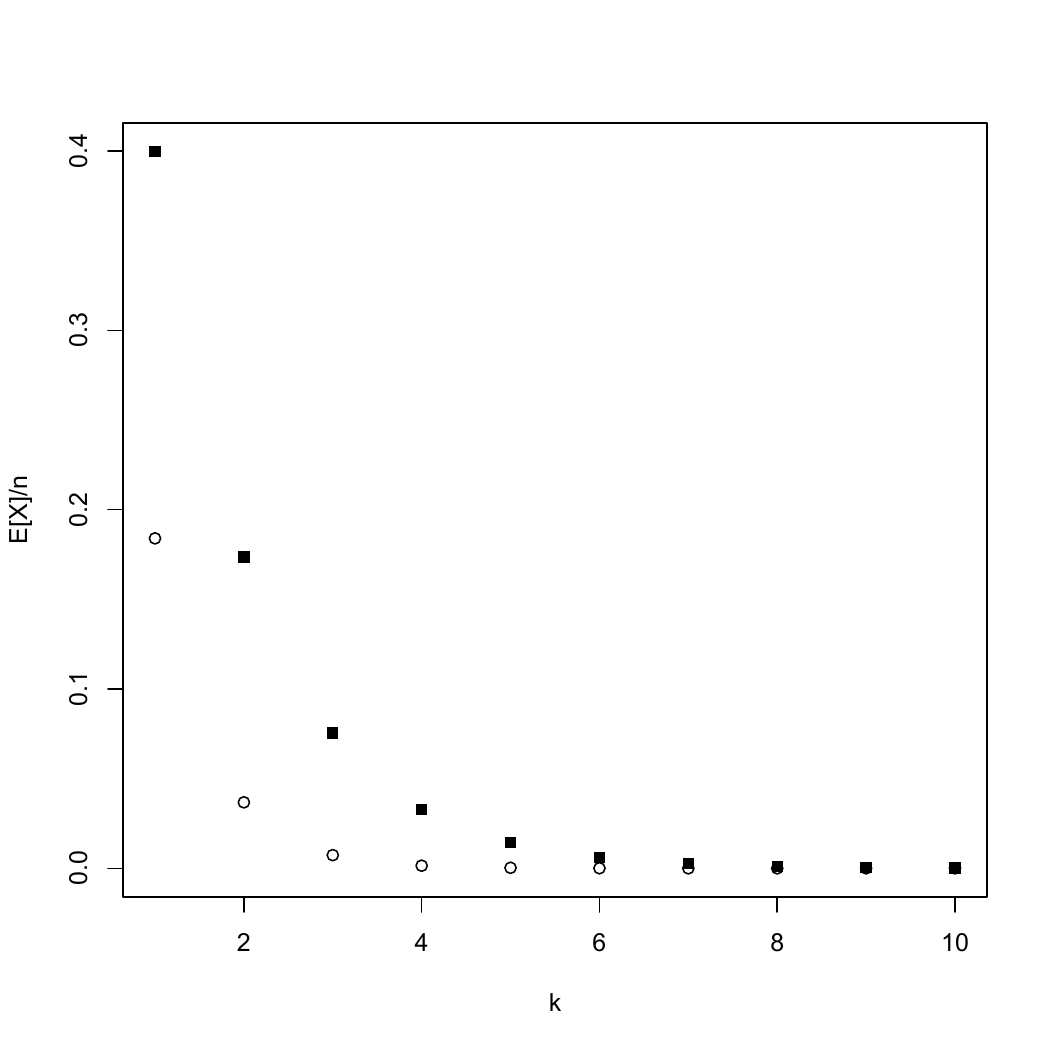}}}
\subfigure[Values of $D$ for various values of $p$]{\scalebox{0.45}{\includegraphics{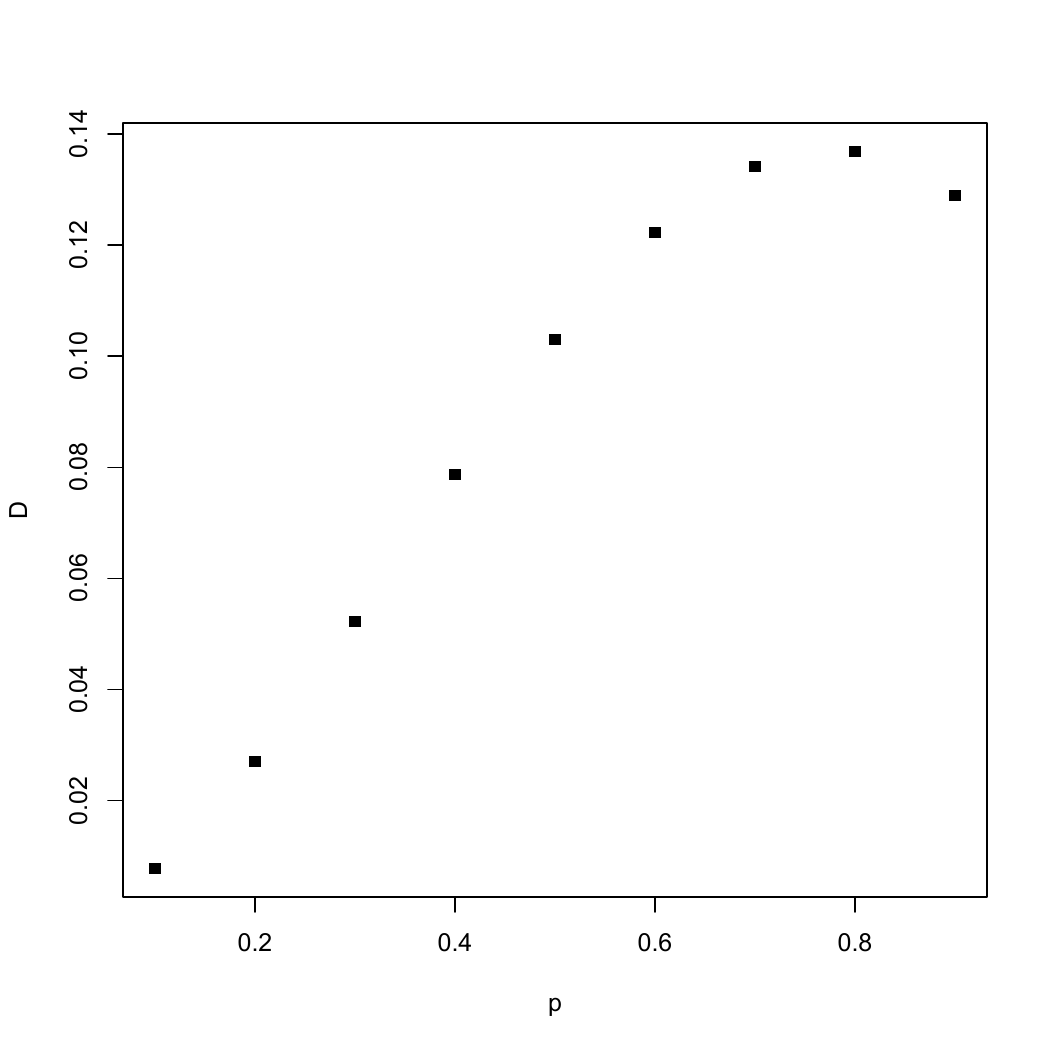}}}
\caption{Plots of the mean number of missed threats per threat for Strategy~A (circle) and Strategy~B (bullet) at various values of $k$ when $p=0.8$, $n=10$,  and $r=1$ $\&$ plot of $D$ for various values of $p$ when $n=10$,  $k=2$,  and $r=1$ in Example~1.}
\label{fig:ex1}
\end{center}
\end{figure}
\end{example}

\begin{theorem}
\label{thm:optimal}
If $ m\geq n\geq 2$ and $0<p<1$,   then Strategy~A is optimal in terms of the mean number of missed threats. 
\end{theorem}

\begin{proof}
Let $X_S$ be the number of missed threats when the interceptors are fired using Strategy~S, where $S$ is an arbitrary strategy.  ($S$ may or may not equal $A$ or $B$.) Without loss of generality,  we may number the $n$ threats as Threat~1, Threat~2, \ldots, Threat~$n$.  For $1\leq i\leq n$, let $m_i$ denote the number of interceptors fired at Threat~$i$.  Then $0\leq m_i\leq n$ and $m_1+\ldots+m_n=m$.  We can write $X_S$ as 
\[
X_S=X_1+\ldots+X_n,
\]
where
\[
X_i=\left\{
		                           \begin{array}{ll}
		                           1,& \hbox{if $m_i=0$,} \\
					 \sim \ber((1-p)^{m_i}) & \hbox{if $m_i\geq 1$,} 
		                          \end{array}
		                         \right.
\]
and $X_i\sim \ber((1-p)^{m_i})$ means $X_i$ is a Bernoulli random variable that takes value $1$ with  probability $(1-p)^{m_i}$ and value $0$ with probability $1-(1-p)^{m_i}$.  Intuitively,  the binary random variable $X_i$ denotes whether Threat $i$ has penetrated the defense layer of the interceptors: when $X_i=1$,  Threat ~$i$  is missed; when $X_i=0$, Threat~$i$ is neutralized.  Our goal is to choose $(m_1, \ldots, m_n)$ so that $\E[X_S]$ is minimized. 

Suppose $m_i=0$ for some $i$. Since $m>n$,  the pigeonhole principle (e.g., Section~28 of Aigner and Ziegler 2018) implies that there is a $j\neq i$ such that $m_j\geq 2$.  Now keep $m_k$ intact for all $k\neq i$ and $k\neq j$,  increase $m_i$ by $1$, and decrease $m_j$ by $1$. The net change of $\E[X_S]$ is
\begin{eqnarray*}
& & [(1-p)^{m_i+1}-(1-p)^{m_i}]+[(1-p)^{m_j-1}-(1-p)^{m_j}]\\
&=& ([1-p)-1]+(1-p)^{m_j-1}[1-(1-p)]\\
&=& p[(1-p)^{m_j-1}-1]\\
&<& 0.
\end{eqnarray*}
This shows that it is never optimal to have $m_i=0$ for any $1\leq i\leq n$. That is,  the optimal strategy must fire at least one interceptor at each of the $n$ threats. 

Next, assume that there are $1\leq i, j\leq n$ such that $m_j\geq 1$ and $m_i-m_j\geq 2$.  If we decrease $m_i$ by $1$, increase $m_j$ by $1$, and keep everything else unchanged, then the net change of $\E[X_S]$ will be
\begin{eqnarray*}
& & [(1-p)^{m_i-1}-(1-p)^{m_i}]+[(1-p)^{m_j+1}-(1-p)^{m_j}]\\
&=& p(1-p)^{m_i-1}- p(1-p)^{m_j} \\
&=& p[(1-p)^{m_i-1}- (1-p)^{m_j}] \\
&<& 0. 
\end{eqnarray*}
Thus,  it is never optimal to have $m_i$ and $m_i$ such that $m_j\geq 1$ and $m_i-m_j\geq 2$. That is, the optimal strategy never fires more than $2$ interceptors at any of the $n$  threats.  

The above two observations,  when taken together,  imply that the optimal strategy is to fire $k+1$ interceptors at each of $r$ threats and $k$ interceptors at each of the remaining $n-r$ threats.  Therefore,  Strategy~A is the optimal strategy. 
\end{proof}

\subsection{When $m< n$}

The formulas for $\E[X_A]$ and $\E[X_B]$ when $m<n$ have not been documented in the literature.   Therefore,  we first derive them first.

\begin{theorem}
\label{thm:mean2}
If $2\leq m<n$ and $0<p<1$,  then
\begin{eqnarray}
\label{eq:mean2}
\E[X_A] &=& n-mp,  \nonumber\\
\E[X_B] &=& n(1-p/n)^m.
\end{eqnarray}
\end{theorem}

\begin{proof}
Similar to the proof of Theorem~\ref{thm:optimal},  we number the $n$ aerial threats as Threat~1, Threat~2, \ldots, Threat~$n$.  Since $m<n$,  at least $n-m$ threats will not be neutralized. Without loss of generality, we may assume that we fire $m$ interceptors over Threat~1, \ldots, Threat~$m$.  For $1\leq i\leq m$, let $m_i$ be the number of interceptors fired at Threat~$i$. Then we can write $X_A$ as
\[
X_A = X_1+\ldots+X_n,
\]
where $X_i\sim \ber((1-p))$ and $X_{m+1}=\ldots X_n=1$. Taking expectation on both sides of the last equation, we obtain 
\[
\E[X_A]=m(1-p)+(n-m)=n-mp.
\]

The derivation of $\E[X_B]$ follows from the same line of reasoning as in the case $m\geq n$; see, for example, Section~2.4.4 of Washburn and Kress (2009).
\end{proof}

\begin{theorem}
If $ 2\leq m< n$ and $0<p<1$,  then $\E[X_A]\leq \E[X_B]$.
\end{theorem}

\begin{proof}
Bernoulli's inequality implies that 
\[
\E[X_B] = n(1-p/n)^m \geq n(1-mp/n)=n-mp=\E[X_A].
\]
\end{proof}

\begin{example}
Suppose  $p=0.9$, $m=18$,  and $n=20$.  Then $\E[X_A]/n=0.190$, $\E[X_B]/n=0.437$,  $\E[X_A]/\E[X_B]=43.53\%$,  and $D=0.247$.  Here Strategy~A is nearly twice as efficient as Strategy~B.  Again, the value of $D$ is significant.  
Panel~(a) of Figure~\ref{fig:ex2} provides plots of $\E[X]/n$ for Strategy~A (circle) and Strategy~B (bullet) when $m$ ranges from $1$ to $19$; Panel~(b) gives the plot of $D$ as $p$ increases.  Similar to Example~1,  when the success probability $p$ is relatively large,  the difference between the two strategies tends to be significant.  

\begin{figure}[h]
\begin{center}
\subfigure[Mean number of missed threats per threat]{\scalebox{0.45}{\includegraphics{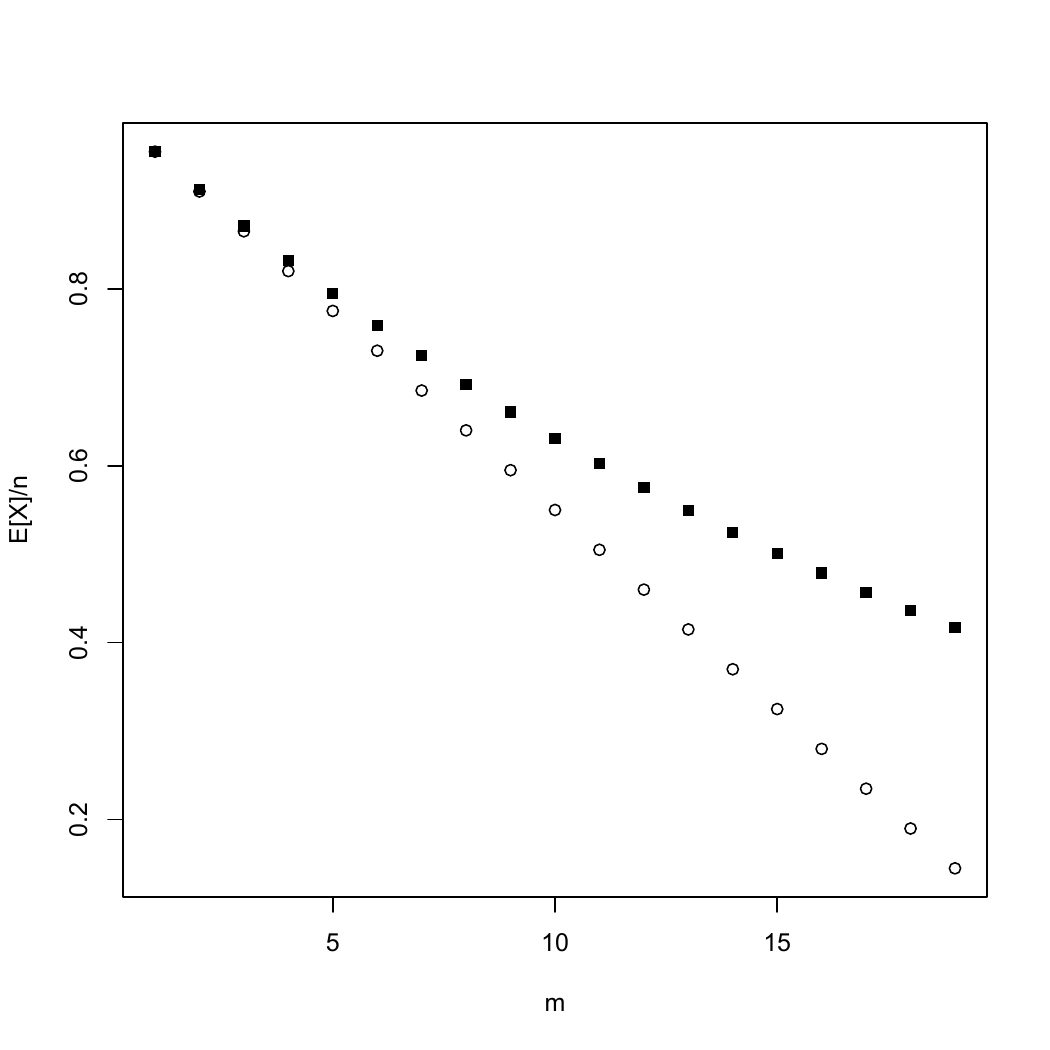}}}
\subfigure[Values of $D$ for various values of $p$]{\scalebox{0.45}{\includegraphics{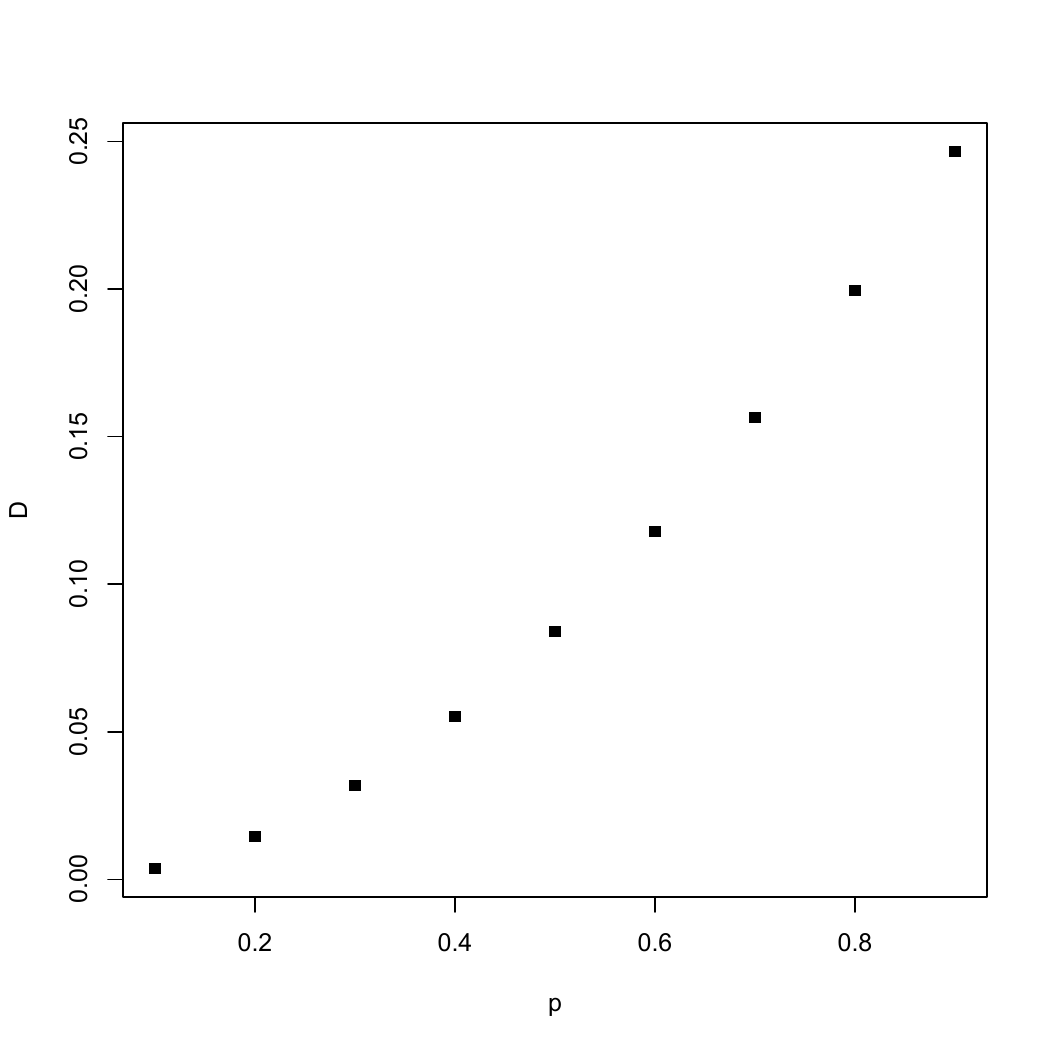}}}
\caption{Plots of the mean number of missed threats per threat for Strategy~A (circle) and Strategy~B (bullet) at various values of $m$ when $p=0.9$ and $n=20$ $\&$ plot of $D$ for various values of $p$ when $n=20$ and $m=18$ in Example~2.}
\label{fig:ex2}
\end{center}
\end{figure}
\end{example}

\begin{theorem}
\label{thm:optimal2}
If $ 2\leq m<n$ and $0<p<1$,   then Strategy~A is optimal in terms of the mean number of missed threats. 
\end{theorem}

\begin{proof}
As in the proof of Theorem~\ref{thm:optimal},  we let $X_S$ be the number of missed threats when the interceptors are fired using Strategy~$S$,  where $S$ is an arbitrary strategy.  We number the $n$ aerial threats as Threat~1, Threat
~2, \ldots, Threat~$n$, and let $m_i$ denote the number of interceptors fired at Threat~$i$ for $1\leq i\leq n$.  We still write $X_S$ as 
\[
X_S=X_1+\ldots+X_n,
\]
where
\[
X_i=\left\{
		                           \begin{array}{ll}
		                           1,& \hbox{if $m_i=0$,} \\
					 \sim \ber((1-p)^{m_i}) & \hbox{if $m_i\geq 1$,} 
		                          \end{array}
		                         \right.
\]
Since $m<n$,  at least $n-m$ threats will not be neutralized,  regardless of the strategy employed.  Thus, we may assume without loss of generality that $X_{m+1}=X_{m+2}=\ldots=X_n=1$.  By the arithmetic-geometric mean inequality (e.g.,  Chapter~2 of Steele 2008 or Theorem~20.II of Aigner and Ziegler 2018),  we have
\begin{eqnarray*}
\E[X_S] &=& \E[X_1]+\ldots+ \E[X_m] +(n-m)\\
		&\geq& \E[X_1]\times \ldots \times \E[X_m]+(n-m), 
\end{eqnarray*}
where the last inequality is an equality if and only if $ \E[X_1]= \E[X_2]=\ldots  \E[X_m]$,  i.e.,  $(1-p)^{m_1}=(1-p)^{m_2}=\ldots (1-p)^{m_m}$.  Since $0<p<1$ and $m_1+\ldots+m_m=m$,  $(1-p)^{m_1}=(1-p)^{m_2}=\ldots (1-p)^{m_m}$ holds if and only if $m_1=\ldots=m_m=1$.  Thus,  $\E[X_S]$ attains its minimum value when we fire exactly one interceptor at a different aerial threat. This proves that Strategy~A is optimal. 
\end{proof}


%



\section*{Conflict of interest}
The author has no conflict of interest to declare.

\section*{References}
\begin{description}

\item{} Aigner, M.~and Ziegler, G.M.~(2018).  \emph{Proofs from THE BOOK}. Springer: New York.

\item{} Atkinson and Kress (2025).  Hard and soft defense against a sequence of aerial threats. \emph{Operations Research}~73(4), 1767--1784.

\item{} Avid, Y.~and Kress, M.~(1997). Evaluating the  effectiveness of shoot-look-shoot tactics in the presence of incomplete damage information.  \emph{Military Operations Research}~3(1), 79--89.


\item{} Glazebrook, K.~and Washburn, A.~(2004).  Shoot-look-shoot: a review and extension.  \emph{Operations Research}~52(3),  454--463.






\item{} Kalyanam, K.~and Clarkson J.~(2021).  Sequential attack salvo size is monotonic nondecreasing in both
time and inventory level.  \emph{Naval Research Logistics}~58(3), 304--321.

\item{} Kline, A.,  Darryl,  A.~and Raymond H.~(2019). The weapon-target assignment problem. \emph{Computers and Operations Research}~105, 226--236.



\item{} Pryluk, R.~and Shima, T.~(2016). Shoot-shoot-look for an air-defense system. \emph{IEEE Systems Journal}~10(1), 151--161.


\item{} Steele, J.M.~(2008). \emph{The Cauchy-Schwarz Master Class: An Introduction to the Art of Mathematical Inequalities}.  Cambridge University Press: Cambridge, UK. 

\item{} Stromberg, K.R.~(1981).  \emph{An Introduction to Classical Real Analysis}. AMS Chelsha Publishing: Providemce, Rhode Island. 





\item{} Washburn, A.~and Kress, M.~(2009). \emph{Combat Modeling}.  Springer: New York. 


\item{} Washburn, A. (2005). The Bang-soak theory of missile attack and terminal defense.  \emph{Military Operations Research}~10(1), 15--23.

\item{} Yost, K.A.~and Washburn, A.~(2000).Optimizing assignment of air-to-ground assets and BDA sensors. \emph{Military Operations Research}~5(2), 77--91.

\end{description}

\end{document}